\documentclass{birkjour}
\usepackage{subfigure}
\usepackage{verbatim}
 \newtheorem{theorem}{Theorem}[section]
 \newtheorem{corollary}[theorem]{Corollary}
 \newtheorem{lemma}[theorem]{Lemma}
 \newtheorem{proposition}[theorem]{Proposition}
 \theoremstyle{definition}
 
 \theoremstyle{remark}
 \newtheorem{remark}[theorem]{Remark}
 \newtheorem{example}{Example}
 \numberwithin{equation}{section}

 \newcommand{\R}{\mathbb{R}}
  \newcommand{\C}{\mathbb{C}}
  \newcommand{\zb}{\overline{z}}

\begin{document}

%
%
%
%---------------------------------------------------------------------------

\title[ Even Dimensional Improper Affine Spheres]
 {Even Dimensional Improper Affine Spheres}

 %----------Author 1
\author[M. Craizer]{Marcos Craizer}

\address{%
Departamento de Matem\'{a}tica - PUC-Rio\br
Rio de Janeiro\br
Brazil}

\email{craizer@puc-rio.br}

\thanks{The first author thanks CNPq, the second author thanks NCN, and the third author thanks Fapesp, for financial support during the preparation of this manuscript.}

%----------Author 2
\author[W. Domitrz]{Wojciech Domitrz}
\address{Faculty of Mathematics and Information Science \br
Warsaw University of Technology \br
ul. Koszykowa 75, 00-662 Warszawa\br
Poland}
\email{domitrz@mini.pw.edu.pl}

%----------Author 3
\author[P. de M. Rios]{Pedro de M. Rios}
\address{Departamento de Matem\'atica - ICMC \br
Universidade de S\~ao Paulo \br
S\~ao Carlos\br
Brazil}
\email{prios@icmc.usp.br}


\subjclass{ 53A15, 53D12}

\keywords{Parabolic affine spheres, Monge-Amp\`ere equation, Special K\"ahler manifolds, Lagrangian submanifolds, Center-chord transform, Exterior Differential Systems, Lagrangian and Legendrian singularities.}

\date{november 15, 2013}

\begin{abstract}
There are exactly two different types of bi-dimensional improper affine spheres: the non-convex ones can be modeled by the center-chord transform of a pair of planar curves while the
convex ones can be modeled by a holomorphic map. In this paper, we show that both
constructions
can be generalized  to arbitrary even dimensions: the former class corresponds to the center-chord
transform of a pair of Lagrangian submanifolds while the latter is related to special K\"ahler manifolds.
Furthermore, we show that the improper affine spheres obtained in this way are solutions of certain exterior differential systems.
Finally, we also discuss the problem of realization of simple stable Legendrian singularities as singularities of these improper affine spheres.
\end{abstract}

%%% ----------------------------------------------------------------------
\maketitle
%%% ----------------------------------------------------------------------

\section{Introduction}

A hypersurface whose Blaschke normal vectors are pointing to a center is called an affine sphere. This class of manifolds is quite large and has been studied by various researchers (\cite{Loftin}).
Hypersurfaces with vanishing affine mean curvature are called affine maximal surfaces and have also been extensively studied (\cite{Aledo09}).
Parabolic, or Improper Affine Spheres (IAS) are affine spheres that are also affine maximal. This is equivalent to saying that the Blaschke normal vectors are parallel, i.e.,
the center of the affine ``sphere'' is at infinity.
There are many articles studying  two dimensional IAS (\cite{Galvez07},\cite{Craizer11},\cite{Craizer12},\cite{Ishikawa06A}, \cite{Martinez05},\cite{Milan13},\cite{Milan14})
and there is also some work in dimension three (\cite{Ishikawa06B}).
In this paper we shall consider IAS in arbitrary even dimensions.

Remind  that for an immersion $\phi: U\subset \R^N\to \R^{N+1}$, if $\tilde{\nabla}$ denotes the canonical affine connection on
$\R^{N+1}$, then any transversal vector field $\xi$ to $\phi(U)$ defines a connection $\nabla$ and symmetric bilinear form $h$ on $TU$ by
$$
\tilde{\nabla}_{X}\phi_*{Y}=\phi_*(\nabla_{X}{Y})+h(X,Y)\xi,
$$
where $X,Y$ are smooth vector fields on $U$. The symmetric bilinear form $h$ defines a volume element on U, denoted $\nu_h$. On the other hand, $\phi^*\Theta_{\xi}$, where
$\Theta_{\xi}(\cdot)=\det(\cdot,\xi)$, defines another volume element on $U$.
Then, a well-known theorem of Blaschke (\cite{Blaschke23},\cite{Nomizu94}) asserts that there exists a unique, up to sign, transversal vector field $\xi$ such that $\nu_h=\phi^*\Theta_{\xi}$ and furthermore $\nabla(\phi^*\Theta_{\xi})=0$. This  unique $\xi$ is called the affine normal, or Blaschke normal vector field to the hypersurface $\phi(U)\subset\R^{N+1}$.

Let $\xi=(0_N,1)\in\R^N\times\R$ be a parallel vector field for the canonical connection $\tilde{\nabla}$ on $\R^{N+1}$. It is well-known (straightforward computation) that the graph of a function $F:V\subset\R^N\to\R$ is an improper affine sphere with affine normal $\xi$ if and only if  $F$ satisfies the classical Monge-Amp\`ere equation
\begin{equation}\label{eq:MongeAmpere}
\det\left(\frac{\partial^2F}{\partial x^2}\right)=c,
\end{equation}
for some constant $c$, where the l.h.s. of (\ref{eq:MongeAmpere}) denotes the Hessian of $F$.
The class of Monge-Amp\`ere equations, in particular the classical one, is an important topic of study in partial differential equations and this highlights the importance of improper affine spheres in geometric analysis (see, e.g., the recent expositions and surveys \cite{LJSX},\cite{Loftin}).

Now, for a smooth function $F:V\to\R$, where $V$ is an open subset of $\mathbb R^{2n}$, we can translate the Monge-Amp\`ere equation \eqref{eq:MongeAmpere} in symplectic terms, as follows. Denote the canonical symplectic form in $\R^{2n}$ by
\begin{equation}\label{eq:omega}
\omega=\sum_{i=1}^n dx_i\wedge dx_{i+n}
\end{equation}
and let $Y_F$ be the
Hamiltonian vector field of $F$, i.e.,
\begin{equation}\label{eq:sympgradient}
dF=\omega(\cdot ,Y_F).
\end{equation}
Then $F$ satisfies the classical Monge-Amp\`ere equation (\ref{eq:MongeAmpere}) if and only if there is a constant $c$ such that
\begin{equation}\label{eq:symplecticMA}
\det(DY_F)=c,
\end{equation}
where $DY_F$ denotes the jacobian matrix of the map $x\mapsto Y_F(x)$.

For an open set  $U\subset\R^{2n}$, consider an immersion $\phi:U\to\R^{2n+1}$ transversal to $\xi=(0,1)\in\R^{2n}\times\R$, where the latter $\R^{2n}$ carries the symplectic form $\omega$. We can write
$\phi(r)=(x(r), f(r))\in \R^{2n}\times\R$, where $x(r)\in V\subset\R^{2n}$ is locally invertible and $f(r)=F(x(r))$, for some $F:V\to\R$.
Denote by $Y_F(x)$ the Hamiltonian vector field of $F$ defined by equation \eqref{eq:sympgradient} and let $y(r)=Y_F(x(r))$. Define
$A(r):T_rU\to T_rU$ by
\begin{equation}\label{eq:definea}
Dy(r)=Dx(r)\cdot A(r).
\end{equation}
It follows from (\ref{eq:symplecticMA}) that $\phi$ is an IAS with Blaschke normal $(0,1)$
if and only if $\det A=c$, for some constant $c$.

In dimension two, any non-convex IAS can be parameterized by asymptotic coordinates and modeled by the center-chord transform of a pair of planar curves (\cite{Craizer08},\cite{Craizer11},\cite{Milan13}).
In this paper, we show that this construction can be generalized to arbitrary even dimensions, where we consider
$x$ as the center and $y$ as the mid-chord of a pair of real Lagrangian submanifolds.
In this case, the asymptotic coordinates condition is replaced by the equation $Dy(r)=Dx(r)\cdot  K_{2n}$, with
\begin{equation}\label{defk}
K_{2n}=\left[
\begin{array}{cc}
-I_n & 0\\
0& I_n
\end{array}
\right],
\end{equation}
where $I_n$ denotes the $n\times n$ identity matrix.

Any convex bi-dimensional IAS can be parameterized by isothermal coordinates and modeled by a holomorphic map (\cite{Galvez07},\cite{Craizer12},\cite{Martinez05}).
This construction can be generalized to arbitrary even dimensions starting from a holomorphic map $G:\C^n\to\C$ and the isothermal condition replaced by
 the relation $Dy(r)=Dx(r)\cdot  J_{2n}$, with
\begin{equation}\label{defj}
J_{2n}=\left[
\begin{array}{cc}
0 & I_n\\
-I_n& 0
\end{array}
\right].
\end{equation}
IAS of this type have already been considered in connection with special K\"ahler manifolds in \cite{Cortes00}, where they were called {\it special}.
We shall see that special IAS are naturally related to a rotated center-chord transform of a pair of complex conjugate Lagrangian submanifolds.

Improper affine spheres can also be seen as geometric solutions of a  Monge-Amp\`ere system (\cite{Ishikawa06A}). Consider a contact form $\theta$ in $\R^{4n+1}$ given by
\begin{equation}\label{eq:theta}
\theta=dz-\sum_{i=1}^n y_{i+n}dx_i-y_idx_{i+n}
\end{equation}
and let
\begin{equation}\label{eq:Omega}
\Omega=\sum_{i=1}^n dx_i\wedge dy_{i+n}+dy_i\wedge dx_{i+n}.
\end{equation}
be the associated canonical symplectic form in $\R^{4n}$.
%COMMENT{If Wojtek wants, we can erase the d\theta.}
For any $F:V\to\R$,
it follows  from (\ref{eq:sympgradient}) that the image of a map $L:V\to\mathbb R^{4n}, \ x\mapsto L(x)=(x,Y_F(x))$, is a Lagrangian submanifold of the symplectic space $(\mathbb R^{4n},\Omega)$, i.e. $L^*\Omega=0$, while the image of a map
${\tilde L}:V\to\mathbb R^{4n+1}, \ x\mapsto {\tilde L}(x)=(x,Y_F(x),F(x))$, is a Legendrian submanifold of the contact space $(\mathbb R^{4n+1},\{\theta=0\})$, i.e. ${\tilde L}^*\theta=0$.
Then, consider the $2n$-form  $\eta$ in $\R^{4n}$ given by
$$
\eta=c\ dx_1\wedge ....\wedge dx_{2n} - dy_1\wedge ....\wedge dy_{2n}.
$$
A solution of the Monge-Amp\`ere system $\{\theta,\eta\}$ is a map $F:V\to\R$ such that ${\tilde L}^*\theta=0$ and $L^*\eta=0$. Thus $F:V\to\R$ is a solution
of this Monge-Amp\`ere system if and only if the graph of $F$ is an IAS.

Differently from the case $n=1$, there are other IAS of dimension $2n$ that are neither center-chord nor special, as we show in some examples. On the other hand, as a main result of this paper, we shall prove that
center-chord and special IAS can be characterized as solutions of certain exterior differential systems (EDS).
Define a symplectic form in $\R^{4n}$ by
\begin{equation}\label{eq:Omega1}
\Omega_1=\sum_{i=1}^n dx_i\wedge dx_{i+n}+dy_i\wedge dy_{i+n}
\end{equation}
and consider the EDS $\mathcal{E}_1=\{\Omega,\Omega_1\}$. We shall verify that the center-chord IAS are exactly
the solutions of the EDS $\mathcal{E}_1$.
Similarly, let
\begin{equation}\label{eq:Omega2}
\Omega_2=\sum_{i=1}^n dx_i\wedge dx_{i+n}-dy_i\wedge dy_{i+n}.
\end{equation}
and define the EDS
$\mathcal{E}_2=\{\Omega,\Omega_2\}$. We shall prove that special IAS are the solutions of $\mathcal{E}_2$.
A natural question  that is left out from this paper is whether there are other classes of IAS that are solutions of some EDS.

Therefore, center-chord and special IAS provide two general classes of solutions to the classical Monge-Amp\`ere equation in any even number of variables.
%COMMENT{ It is clear from what we are doing.}YES, BUT THIS IS ADVERTISEMENT! IN ADVERTISEMENT, WE REPEAT!
But  general solutions are known to present singularities. In fact, except for paraboloids, any  convex IAS admits singularities, thus singularities appear naturally in the context of improper affine spheres.

Denote by $\pi_1:\R^{2n}\times\R^{2n}\to\R^{2n}$ the projection $\pi_1(x,y)=x$ and by ${\mathcal L}$ the image of the Lagrangian immersion $L$ described above. In the context of IAS, singularities of the Lagrangian map $\pi_1:{\mathcal L}\to\R^{2n}$ are the ones which were called admissible in \cite{Martinez05} and, in that paper, IAS with admissible
singularities were called Improper Affine Maps.  In dimension $2$, admissible singularities have been well studied (\cite{Galvez07},\cite{Craizer11},\cite{Craizer12},\cite{Martinez05},\cite{Milan13},\cite{Milan14}).

One can also consider singularities of the Legendrian map $\pi_2:{\tilde{\mathcal L}}\to\R^{2n+1}$, where
$\pi_2:\R^{2n}\times\R^{2n}\times\R\to\R^{2n}\times\R$ is the projection $\pi_2(x,y,z)=(x,z)$ and
${\tilde {\mathcal L}}$ is the image of the above Legendrian immersion ${\tilde L}$.
For $2$-dimension IAS, these singularities were studied in \cite{Ishikawa06A}.

We shall study in this paper the stable singularities of the above Lagrangian and Legendrian maps for general even dimensions. From  Theorem 4.1 in  \cite{Domitrz13}, we know that any simple stable Lagrangian singularity is realizable as a center-chord IAS.
Here we prove that this also holds for special IAS,
and our proof extends naturally to the Legendrian setting, showing that any simple  stable Lagrangian and Legendrian singularity is realizable as a special IAS.
 Starting from \cite{Domitrz13}, one can easily verify that the Legendrian statement also holds  for center-chord IAS, but here we prove this explicitly in a way that highlights the similarities between the center-chord and the special IAS. In the center-chord case, we also comment on the boundary singularities, or ``on-shell'' singularities, that appear in the limit of vanishing chords and have a special symmetry, as described in \cite{DMR}.

This paper is organized as follows: In section 2 we establish the notation and describe the symplectic condition for an immersion to be an IAS.
In section 3 we describe the models for center-chord and special IAS. In section 4 we prove that the center-chord and special IAS are the solutions of the EDS $\mathcal{E}_1$ and $\mathcal{E}_2$, respectively.
In section 5 we discuss the Lagrangian and Legendrian singularities of these maps.

{\it Acknowledgements}: The third author benefitted from the hospitality of the Mathematics Department at UC Berkeley during the final stages of the preparation of this manuscript. Special thanks to his host, Alan Weinstein, for comments and suggestions that improved this final version.

\section{Symplectic characterization of IAS}

\subsection{The symplectic structructure of $TV$ and contact structure of $TV\times \R$}

Let $V$ be an open subset of $\mathbb R^{2n}$ and let $\omega$ be the canonical symplectic form on $V$.
A map $\flat: TV\ni v\mapsto \omega(v,\cdot) \in T^{\ast}V$ is a isomorphism of the bundles $TV, \ T^{\ast}V$.
Let $\alpha$ be the canonical Liouville $1$-form on $T^{\ast}V$.
Then  $\Omega=\flat^{\ast}d\alpha$ is a symplectic form on $TV$ and $\theta=dz+\flat^{\ast}\alpha$ is a contact form on $TV\times \mathbb R$, where $z$ is a coordinate on $\mathbb R$.

Let $F:V\rightarrow \mathbb R$ be a smooth function. Let $Y_F$ be the Hamiltonian vector field of $F$ e. i. $\omega(Y_F, \cdot) = -dF(\cdot)$.

\begin{proposition}
A map $\tilde{L}:V\ni x\mapsto (x,Y_F(x),F(x))\in TV\times \mathbb R$ is a Legendrian immersion to the contact space $(TV\times \mathbb R, \{ \theta=0 \})$.
\end{proposition}
\begin{proof}
It is obvious that $\tilde{L}$ is a immersion. We have
$$\tilde{L}^{\ast}\theta=dF+\tilde{L}^{\ast}\flat^{\ast}\alpha=dF+(\flat\circ\tilde{L})^{\ast}\alpha).$$
On the other hand $\flat\circ\tilde{L}=\flat(Y_F)=\omega(Y_F,\cdot)=-dF$. By the tautological property of the Louville $1$-form $\alpha$ we have $(\beta)^{\ast}\alpha=\beta$ for any $1$-form $\beta$ on $V$. Thus we get $\tilde{L}^{\ast}\theta=dF+(-dF)^{\ast}\alpha=dF-dF=0$.
\end{proof}

Using the same arguments one can prove the following proposition.
\begin{proposition}\label{LagHam}
A map $L:V\ni x\mapsto (x,Y_F(x))\in TV$ is a Lagrangian immersion to the symplectic space $(TV, \Omega)$.
\end{proposition}

Let $x=(x_1,\cdots,x_{2n})$ be a coordinate system on $V$ and
\begin{equation}\label{omegacan}
\omega=\sum_{i=1}^{n}dx_i\wedge dx_{i+n}.
\end{equation}
Let $(x,y)=(x_1,\cdots,x_{2n},y_1,\cdots,y_{2n})$ be the standard coordinate system on $TV$, $(x,y,z)$ be a coordinate system on $TV\times \mathbb R$ and finally let $(x,p)=(x_1,\cdots,x_{2n},p_1,\cdots,p_{2n})$ be the standard coordinate system on $T^{\ast}V$.

The Liouville $1$-form in these coordinates is $\alpha=\sum_{i=1}^n p_i dx_i$ and the isomorphism is
$$\flat(x,y)=(x_1,\cdots,x_{2n},-y_{n+1},\cdots,-y_{2n},y_1,\cdots,y_n).$$
Thus the symplectic form and the contact form have the following forms
\begin{equation}\label{Omega}
\Omega=\sum_{i=1}^n dx_i\wedge dy_{i+n} +dy_i\wedge dx_{i+n},
\end{equation}
\begin{equation}\label{theta}
\theta=dz-\sum_{i=1}^ny_{i+n}dx_i-y_idx_{i+n}.
\end{equation}
\subsection{Center-chord transforms}

The form $\Omega$ is also called the tangential lift of $\omega$ and, under the identification  $y=\dot x$, it can be formally identified with the ``time derivative of $\omega$'' and is often  denoted by $\dot\omega$. By Proposition \ref{LagHam}, the ``graph'' of a Hamiltonian function $F$, i.e. its vector field $Y_F$, is Lagrangian w.r.t. $\Omega$. Such a Hamiltonian vector field is usually seen as the generator of a canonical transformation on $(V,\omega)$.

However, the form $2\Omega$ can also be seen as the pullback of the ``difference'' symplectic form  $\omega\ominus\omega = (\omega,-\omega)$ on $V\times V$ via the linear diffeomorphism
\begin{equation}\label{psi}\Psi:TV\to V\times V , \  (x,y)\mapsto(x+y,x-y)=(x_+,x_-).\end{equation} In this context, the coordinates $(x,y)$ are called the center and mid-chord coordinates and $\Psi^{-1}$ is the {\it center-chord transform}. This is globally well-defined when $V=\mathbb R^{2n}$ and thus, in such a case,
\begin{equation}\label{invpsi} \Psi^{-1}:V\times V\to TV , \ (x_+,x_-)\mapsto\left(\frac{x_++x_-}{2},\frac{x_+-x_-}{2}\right)=(x,y),\end{equation}
so that
\begin{equation}\label{pullbackOmega} (\Psi^{-1})^*\Omega=\frac{1}{2}(\omega_+-\omega_-), \end{equation}
where $\omega_+$ and $\omega_-$ are given as in (\ref{omegacan}), for coordinates $x_+=(x^+_i, ..., x^+_{2n})$, $x_-=(x^-_1, ..., x^-_{2n})$ in $V\times V$.

Then, a pair of  real Lagrangian submanifolds, $\Lambda_1,\Lambda_2$, of $(V , \omega)$ pulls back to a real Lagrangian submanifold $L=\Psi^{-1}(\Lambda_1\times\Lambda_2)$ of  $(TV,\Omega)$ which, when projecting regularly to the center subspace $V\ni x$  can be described as the ``graph'' of a function $F$ by $L(x)=(x,Y_F(x))$, as above  (here the center subspace $V\simeq T^0V$ is seen as  the zero section of $TV$).

In this setting, the function $F$ is the Poincar\'e, or center generating function of the canonical relation $\Lambda_1\times\Lambda_2\subset V\times V$ (\cite{Poin}\cite{Wein}\cite{RO}).  Note that this differs from the usual setting when $Y_F$ is  the Hamiltonian vector field that generates an infinitesimal canonical transformation $\Phi:V\to V$ because, in the latter case, although the graph of $\Phi$ is also a real Lagrangian submanifold of $V\times V$, it projects regularly to both copies of $V$.

The above  center-chord description can be generalized to study  complex Lagrangian submanifolds  of a complexified (real) symplectic vector space $(V^{\C},\omega)$. In this case, one fixes a complex structure and assigns a pair of holomorphic and anti-holomorphic coordinates, $x=(u,\bar u)\in V^{\C}$, so that the symplectic form is given in these complex canonical coordinates by
\begin{equation}\label{complexomega} \omega=\frac{i}{2}du\wedge d\bar u \ ,\end{equation}
with index summation subtended,
and thus $\omega$ is still a real form,  $\bar\omega=\omega$.

The map $\Psi:TV^{\C}\to V^{\C}\times V^{\C}$, given by (\ref{psi}), assigns complex canonical coordinates  in each copy of $V^{\C}$ which are induced from complex canonical coordinates $x=(u,\bar u), y=(w,\bar w)$ in $TV^{\C}$, and vice versa, by
\begin{equation}\label{complexpm} x_+=(z_+,\bar z_+) =(u+w,\bar u+\bar w) \ , \ x_-=(z_-,\bar z_-) =(u-w,\bar u-\bar w) \ , \end{equation}
and thus the relevant {\it real} symplectic forms in  $TV^{\C}$ and  $V^{\C}\times V^{\C}$ are written in these complex canonical coordinates as
\begin{eqnarray} & \Psi^*(\omega_+-\omega_-)=2\Omega={i}(dw\wedge d\bar u + du\wedge d\bar w) \ , \\
& (\Psi^{-1})^*(2\Omega)=\omega_+-\omega_-=\frac{i}{2}(dz_+\wedge d\bar z_+ - dz_-\wedge d\bar z_-) \ . \end{eqnarray}

However, for various reasons, some to be made clearer further below, it is also useful to define the {\it rotated center-chord transform} as
\begin{equation}\label{complexinvpsi}
\widetilde{\Psi}^{-1}:V^{\C}\times V^{\C}\to TV^{\C}, \ (\tilde x_+,\tilde x_-)\mapsto\left(\frac{\tilde x_++\tilde x_-}{2},\frac{\tilde x_+-\tilde x_-}{2i}\right)=(x,y),
\end{equation}
with inverse
\begin{equation}\label{complexpsi}
\widetilde{\Psi}: TV^{\C}\to V^{\C}\times V^{\C}, \ (x, y)\mapsto (x+i y, x-i y)=(\tilde x_+,\tilde x_-).
\end{equation}
Note that the new map $\widetilde{\Psi}$ is obtained from the old one by first multiplying each fiber of $TV^{\C}$ by $i$, that is:
\begin{equation}\label{multii} J_x : T_xV^{\C}\to T_xV^{\C} \ , \ y\mapsto iy \ , \end{equation}
but this is equivalent to  performing a $\pi/2$ rotation on each fiber of $TV^{\C}$, so that $J_x$ can also be defined using only  real coordinates, that is:
\begin{equation}\label{multiir} J_x : T_xV\to T_xV \ , \  J_x^2=-Id_x \ . \end{equation}
Now, if $ J$ denotes the map $TV^{\C}\to TV^{\C}$ (or $TV\to TV$) which is defined by the collection of  fiber maps $J_x$ as above, $\forall x\in V$,  then
\begin{equation}\label{compJ} \widetilde{\Psi}=\Psi\circ J \ , \end{equation}
so that $\widetilde\Psi$ and its inverse, the rotated center-chord transform $\widetilde\Psi^{-1}$, can also be defined as real maps $TV\to V\times V$ and $V\times V\to TV$, respectively.

\subsection{Immersions that are transversal to a constant direction}

In this section we recall some basic facts concerning dual connections.
Let $U\subset\R^{2n}$ be an open simply connected set and let $V$ be an open set of the symplectic affine space $\R^{2n}$ with its canonical symplectic form $\omega$.

Consider an immersion $\phi:U\to V\times\R\subset\R^{2n}\times\R$ transversal to $\xi=(0,1)$. For $r\in U$, write
\begin{equation}\label{eq:graphconnection}
\tilde\nabla_{X}\phi_*{Y}=\phi_*(\nabla_{X}{Y})+h(X,Y)\xi,
\end{equation}
for any smooth vector fields $X,Y$ on $U$, where $\tilde\nabla$ denotes the canonical connection in $\R^{N+1}$.

\begin{lemma}
$\nabla$ is a torsion free affine connection and $h$ is a symmetric bilinear form.
\end{lemma}
\begin{proof}
Interchanging the roles of $X$ and $Y$ in \eqref{eq:graphconnection} and subtracting we obtain
$$
\nabla_{X}{Y}-\nabla_{Y}X-[X,Y]=0
$$
and
$$
h(X,Y)-h(Y,X)=0.
$$
\end{proof}

Denote $\phi(r)=(x(r), f(r))$ with $f(r)=F(x(r))$. Denote by $Y_F$ denote the Hamiltonian vector field of $F:V\subset\mathbb R^{2n}\to \mathbb R$ and let
$y(r)=Y_F(x(r))$. We have that
\begin{equation}\label{defin1}
df(r)\cdot u=\omega(Dx(r)u,y(r)), \ \ \  \forall u\in T_rU,
\end{equation}
where the dot $\cdot$ in the l.h.s. of  (\ref{defin1}) denotes the usual vector-form contraction, seen also as the matrix product of a line
$1\times(2n)$ and a column $(2n)\times 1$ matrix, as we shall be using the dot $\cdot$ to denote matrix product in various places below.
Fix a basis $\left\{e_i\right\}_{1\leq i\leq 2n}$ of $T_rU$ and write
$$
x_{r_i}=Dx(r)\cdot e_i;\ \ \ x_{r_ir_j}= Dx_{r_i}\cdot e_j.
$$

\begin{lemma}
 We have that
\begin{equation}\label{eq:hsym}
h\left(e_i,e_j \right)=\omega(x_{r_i},y_{r_j})=\omega(x_{r_j},y_{r_i})
\end{equation}
and the $\nabla$-Christoffel symbols $\Gamma_{ij}^k$ are given by the following formula
\begin{equation}\label{eq:Christoffel1}
x_{r_ir_j}=\sum_{k}\Gamma_{ij}^k x_{r_k}.
\end{equation}
%$$
%\nabla_{e_k}e_i=\sum_l\Gamma_{ik}^l e_l,
%$$
%where
%\begin{equation}\label{eq:Christoffel1}
%x_{r_kr_i}=\sum_l\Gamma_{ik}^lx_{r_l},
%\end{equation}
\end{lemma}
\begin{proof}
Since
\begin{equation*}
\phi_{r_i}=\left(x_{r_i},\omega(x_{r_i},y)\right).
\end{equation*}
we obtain
\begin{equation}
\phi_{r_ir_j}=\left(x_{r_ir_j},\omega(x_{r_ir_j},y)\right)+\left(0,\omega(x_{r_i},y_{r_j})\right).
\end{equation}
Now observe that the first parcel in the second member is tangent while the second parcel is a multiple of $\xi$. On the other hand we have
$$
\left(x_{r_ir_j},\omega(x_{r_ir_j},y)\right)=\phi_{\ast}(\nabla_{e_j}e_i)=\sum_{k}\Gamma_{ij}^k\phi_{r_k}=\left(\sum_{k}\Gamma_{ij}^k x_{r_k}, \omega(\sum_{k}\Gamma_{ij}^k x_{r_k},y)\right)
$$
Thus the lemma is proved.
\end{proof}

Define $g:U\to\R$ by
\begin{equation}
dg(r)\cdot u=\omega(Dy(r)u,x(r)), \ \ \ \forall u\in T_rU.
\end{equation}
Assuming $y(r)$ is locally invertible, the immersion $\psi(r)=(y(r),g(r))$ is called the {\it dual immersion} of $\phi$ and the function $G$ such that
$g(r)=G(y(r))$, is the {\it Legendre transform} of $F$. Equation \eqref{eq:hsym} implies that $g$ is locally well-defined.

Denoting by $\overline{\nabla}$ and $\overline{h}$ the connection and metric of the dual immersion, we have that $\overline{h}=h$ and the $\overline{\nabla}$-Christoffel symbols
${\overline\Gamma_{ij}^k}$ are given by the following formula
\begin{equation}\label{eq:Christoffeldual1}
y_{r_ir_j}=\sum_k {\overline\Gamma_{ij}^k} y_{r_k}.
\end{equation}
%\begin{equation}
%{\overline\nabla}_{e_k} e_j=\sum_l{\overline \Gamma}_{kj}^l e_l,
%\end{equation}
%where
%\begin{equation}\label{eq:Christoffeldual1}
%y_{r_kr_j}=\sum_l{\overline \Gamma}_{kj}^ly_{r_l}.
%\end{equation}

\begin{lemma}
$\overline{\nabla}$ is the $h$-dual, or metric-dual of $\nabla$. In other words, the connection $\hat{\nabla}$ of the metric $h$ is given by
$$
\hat{\nabla}=\frac{\nabla+{\overline\nabla}}{2}.
$$
\end{lemma}
\begin{proof}
We must prove that
\begin{equation}\label{eq:dualconnection}
\frac{\partial}{\partial r_k}h\left( e_i, e_j \right) =h\left( \nabla_{e_k} e_i , e_j\right)+h\left(e_i,  {\overline \nabla}_{e_k} e_j  \right).
\end{equation}
The first member of \eqref{eq:dualconnection} is equal to
$$
\omega\left(x_{r_kr_i}, y_{r_j}\right)+\omega\left(x_{r_i}, y_{r_kr_j}\right)
=\omega\left(\sum_l\Gamma_{ik}^lx_{r_l}, y_{r_j}\right)+\omega\left(x_{r_i}, \sum_l{\overline \Gamma}_{kj}^l y_{r_l}\right)
$$
$$
=\omega\left(x_*\left(\nabla_{e_k} e_i\right), y_*e_j\right)+\omega\left(x_*e_i, y_*\left({\overline \nabla}_{e_k} e_j \right)\right),
$$
which is exactly the second member of \eqref{eq:dualconnection}.
\end{proof}

Denote by $A(r):T_rU\to T_rU$ the invertible linear map satisfying the condition $Dy(r)=Dx(r)\cdot A(r)$. We shall make no distinction between the linear map $A(r)$
and its matrix $A(r)=\left(a_{ij}(r)\right)_{i,j=1,\cdots,2n}$ in the canonical basis $\{e_1,..,e_{2n}\}$.

\begin{proposition}\label{lemma:aconstant}
We have
${\overline\nabla}=A^{-1}\nabla A$.

\end{proposition}
\begin{proof}
We must prove that
$$
A{\overline\nabla}_{e_k}{e_j}=\nabla_{e_k}(Ae_j),
$$
for any $1\leq j,k\leq 2n$. This is equivalent to the following formula
\begin{equation}\label{eq:Christoffel2a}
\sum_{l}\overline{\Gamma}_{ik}^{l}a_{sl}=\sum_{j}a_{ji}\Gamma_{jk}^s+\frac{\partial a_{si}}{\partial r_k},
\end{equation}
for any $1\leq l\leq 2n$. But
\begin{equation}\label{eq:yxaexplicit}
y_{r_i}=\sum_{j}x_{r_j}a_{ji}
\end{equation}
Differentiating \eqref{eq:yxaexplicit} with respect to $r_k$ we obtain
\begin{equation}\label{eq:yxaconstant}
y_{r_ir_k}=\sum_{j}x_{r_jr_k}a_{ji}+\frac{\partial a_{ji}}{\partial r_k}x_{r_j}.
\end{equation}
Applying \eqref{eq:yxaexplicit} and \eqref{eq:yxaconstant} in \eqref{eq:Christoffeldual1} and using  \eqref{eq:Christoffel1} we obtain \eqref{eq:Christoffel2a}.
\end{proof}

In the next proposition we present a sufficient condition for  $A$ to be parallel with respect
to the metric connection $\hat{\nabla}$.

\begin{proposition}\label{prop:aparallel}
If $\nabla \left(A^2\right)=0$ then $\hat{\nabla} A=0$.
\end{proposition}
\begin{proof}
We have that
\begin{equation*}
(\hat{\nabla}_XA)Y=\hat{\nabla}_X(AY)-A\hat{\nabla}_XY
\end{equation*}
\begin{equation*}
=\frac{1}{2}\left(  \nabla_X(AY)+\overline{\nabla}_X(AY) -A\nabla_XY-A\overline{\nabla}_XY    \right).
\end{equation*}
Now Proposition \ref{lemma:aconstant} implies that
\begin{equation*}
(\hat{\nabla}_XA)Y=\frac{1}{2}\left(  \nabla_X(AY)+A^{-1}{\nabla}_X(A^2Y) -A\nabla_XY-{\nabla}_X(AY)    \right)
\end{equation*}
\begin{equation*}
=\frac{1}{2}\left( A^{-1} \nabla_X(A^2Y)-A{\nabla}_XY  \right)
\end{equation*}
If $\nabla \left(A^2\right)=0$ then $\nabla_X(A^2Y)=A^2\nabla_X Y$ and this last expression vanishes.
\end{proof}

In this paper we are specially interested in the cases $A(r)=K_{2n}$ and $A(r)=J_{2n}$. Next result which is a corollary of Proposition \ref{prop:aparallel} shows that in this case $A$ is parallel with respect
to the metric connection $\hat{\nabla}$.

\begin{corollary}\label{prop:kj-aparallel}
If $A(r)=K_{2n}$ or $A(r)=J_{2n}$ then $\hat{\nabla} A=0$.
\end{corollary}

\subsection{Improper affine spheres}

An immersion $\phi:U\to\R^{2n+1}$ is an improper affine sphere with Blaschke normal vector $\xi=(0,1)$ if the volume determined by the metric $h$ coincides
with the volume $\phi^*\Theta_{\xi}$, where $\Theta_{\xi}(\cdot)=\det(\cdot,\xi)$ (see \cite{Nomizu94}). This is equivalent to $|\det(h)(r)|=det(Dx)^2(r)$, for any $r\in U$.

Let
\begin{equation}\label{defineB}
B(r)=Dx(r)\cdot A(r)\cdot Dx(r)^{-1}.
\end{equation}

\begin{lemma}
We have that
$$
\det(h)=\det(Dx)^2\det A.
$$
\end{lemma}
\begin{proof}
The symmetric matrix $h$ has entries
$$
h_{ij}=\omega(Dx(r)\cdot e_i, Dx(r)\cdot A(r)\cdot e_j).
$$
Since $B(r)=Dx(r)\cdot A(r)\cdot Dx(r)^{-1}$, we have that $\det B=\det A$ and
$$
h_{ij}=\omega(Dx(r)\cdot e_i,B(r)\cdot Dx(r)\cdot e_j).
$$
In terms of matrices, $h=Dx(r)^t \cdot {J_{2n}}\cdot B(r)\cdot Dx(r)$. Hence
$$
\det(h)=\det(Dx)^2\det B,
$$
thus proving the lemma.
\end{proof}

\begin{corollary}
The metric $h$ is non-degenerate if and only if $A$ is invertible.
\end{corollary}

\begin{proposition}\label{prop:Deta}
$\phi$ is an improper affine sphere if and only if $\det A$ is constant.
\end{proposition}
\begin{proof}
The immersion $\phi$ is an improper affine sphere if and only if the metric volume in the tangent space is the same
as the volume determined by $\xi$. The metric volume is $\sqrt{\det(h)}$, while the volume
determined by $\xi$ is
$$
\det(\phi_{r_1},...,\phi_{r_{2n}},\xi)=\det(x_{r_1},...,x_{r_{2n}})=\det(Dx).
$$
Thus $\phi$ is an improper affine sphere if and only if  $\sqrt{\det(h)}=c\det(Dx)$, for some constant $c$. Since
$$
\sqrt{\det(h)}=\det(Dx) \sqrt{\det A},
$$
the proposition is proved.
\end{proof}

\section{Two distinguished classes of even dimensional improper affine spheres}

In this section we shall describe two classes of even dimensional improper affine spheres. The first one is obtained by taking $x$
as the center and $y$ as the mid-chord of a pair of points of a given pair of real Lagrangian submanifolds.
It is a natural generalization of the class of bi-dimensional improper affine spheres with indefinite metric.
The second one is a natural generalization of the class of bi-dimensional improper affine spheres with definite metric. IAS in this latter class are called special (\cite{Cortes00}).

In the center-chord case, the matrix $A$ is similar to $K_{2n}$, while in the special case the matrix $A$ is similar to $J_{2n}$.
By proposition \ref{prop:aparallel}, in both cases the matrix $A$ is parallel with respect to the metric connection. This fact was proved in \cite{Cortes00}
in the special IAS case.

\subsection{Center-chord improper affine spheres}\label{sec:Center-chord}

Let $U_1$, $U_2$ be open subsets of $\R^n$ such that $U=U_1\times U_2\subset\R^{2n}$ is simply connected. Let $\beta:U_1\to\R^{2n}$, $\gamma:U_2\to\R^{2n}$ be real Lagrangian embeddings and $\Lambda_1=\beta(U_1)$,   $\Lambda_2=\gamma(U_2)$.

Define the center $x:U\to\R^{2n}$ by
$$
x(s,t)=\frac{1}{2}\left(\beta(s)+\gamma(t)\right)
$$
and the half-chord $y:U\to\R^{2n}$ by
$$
y(s,t)=\frac{1}{2}\left( \gamma(t)-\beta(s) \right),
$$
where $s=(s_1,...,s_n)\in U_1$ and $t=(t_1,...,t_n)\in U_2$. Observe that since $\beta$ and $\gamma$ are Lagrangian,
$$
\omega(x_{s_i},y_{s_j})=\omega(x_{t_i},y_{t_j})=0.
$$
Moreover,
$$
\omega(x_{s_i},y_{t_j})=\omega(\beta_{s_i},\gamma_{t_j})=\omega(\gamma_{t_j},-\beta_{s_i})=\omega(x_{t_j},y_{s_i}),
$$
which implies in the existence of some function $f:U\to\R$ satisfying
$$
f_{s_i}=\omega(x_{s_i},y),\ \ f_{t_i}=\omega(x_{t_i},y), \text{for} \ i=1,\cdots,n .
$$

\begin{theorem}\label{ccias}
Assume that the tangent spaces of $\Lambda_1$ at $\beta(s)$ and of $\Lambda_2$ at $\gamma(t)$ are transversal. Then the immersion $\phi(s,t)=(x(s,t),f(s,t))$ is an immersion with $A(r)=A(s,t)=K_{2n}$. As a consequence, $\Sigma^{2n}= \text{Image}(\phi)\subset\mathbb R^{2n+1}$ is an improper affine sphere with Blaschke normal $\xi=(0_{2n},1)$ and Blaschke metric given by
\begin{equation}\label{eq:ccmetric}
h=\frac{1}{4}\left[
\begin{array}{cc}
0 & \omega(\beta_{s_i},\gamma_{t_j}) \\
 \omega(\beta_{s_i},\gamma_{t_j}) & 0
\end{array}
\right].
\end{equation}
\end{theorem}

\begin{proof}
The first statement follows from
$$
y_{s_i}=-\frac{1}{2}\beta_{s_i}=-x_{s_i};\ \ y_{t_j}=\frac{1}{2}\gamma_{t_j}=x_{t_j}.
$$
Thus by Proposition \ref{prop:Deta},  $\phi$ is an improper affine sphere with Blaschke metric given by equation \eqref{eq:hsym}.
\end{proof}

The function $f(s,t)$ can be geometrically interpreted as follows: Fix points $\beta(s_0)\in\Lambda_1$ and $\gamma(t_0)\in\gamma$ and take curves $\tilde{\beta}\subset\Lambda_1$
connecting $\beta(s_0)$ with $\beta(t)$ and $\tilde{\gamma}\subset\Lambda_2$ connecting $\gamma(t)$ with $\gamma(t_0)$. Denote by $S$ a $2$-surface
whose boundary is the concatenation of the chord $\gamma(t_0)\beta(s_0)$, $\tilde{\beta}$, the chord $\beta(s)\gamma(t)$ and $\tilde{\gamma}$. Then $f(s,t)$ is the area of $S$.
Observe that the Lagrangian condition for $\Lambda_1$
and $\Lambda_2$ implies that this area does not depend on the choice of $\tilde{\beta}$ and $\tilde{\gamma}$.

Under the transversality hypothesis of Theorem \ref{ccias}, the projection $\pi: T\mathbb R^{2n}\to\mathbb R^{2n}$ restricted to $\{(x(s,t),y(s,t)):(s,t)\in U\}=Y_F(V)\subset T\mathbb R^{2n}$ is regular and therefore $ f(s,t)=F(x(s,t))$, where the function $F:V\subset\mathbb R^{2n}\to \mathbb R$ is the center generating function of  $\Lambda_1\times \Lambda_2$ and satisfies the classical Monge-Amp\`ere equation, $\det[\partial^2F]=c$, for some constant $c$.

These improper affine spheres that are naturally related to the center-chord transform of a pair of real Lagrangian submanifolds, and its center generating function, shall be called {\it center-chord improper affine spheres}.
Singularities of these improper affine spheres (or equivalently of this class of solutions to the classical Monge-Amp\`ere equation) occur when the transversality hypothesis fail, and these shall be studied in section \ref{singularsection}.

\subsection{Special improper affine spheres}

Let $U$ be open subset of $\C^n$. Given a holomorphic map $H:U \to\C$, we write
\begin{equation}\label{holom-H}
H(z)=\tilde P(z,\bar z)+i\tilde Q(z, \bar z),
\end{equation}
with $\tilde P,\tilde Q:U\to\R$. Then, for $z=s+it$, $z=(z_1,..,z_n)$, $s=(s_1,..,s_n)$, $t=(t_1,..,t_n)$, we define
\begin{equation}\label{hol2}
P(s,t)=\tilde P(s+it,s-it), \ Q(s,t)=\tilde Q(s+it,s-it).
\end{equation}
Hence $\frac{\partial Q}{\partial s}=(\frac{\partial Q}{\partial s_1},...,\frac{\partial Q}{\partial s_n} )$ and
$\frac{\partial Q}{\partial t}=(\frac{\partial Q}{\partial t_1},...,\frac{\partial Q}{\partial t_n} )$. In this setting, we define
$$
x(s,t)=(x^{(1)}(s,t),x^{(2)}(s,t))=\left(s, \frac{\partial Q}{\partial t}  \right)
$$
and
$$
y(s,t)=(y^{(1)}(s,t),y^{(2)}(s,t))=\left(t, \frac{\partial Q}{\partial s}\right).
$$
Let
$$
f(s,t)= Q(s,t)-\sum_{k=1}^n t_k\cdot \frac{\partial Q}{\partial t_k}(s,t).
$$
A straightforward calculation shows that equation \eqref{defin1} holds.

\begin{theorem}
At points $(s,t)$ such that \begin{equation}\label{transvspecial} \det(\frac{\partial^2 Q}{\partial t^2})\neq 0 \ ,\end{equation} the map $\phi(s,t)=(x(s,t),f(s,t))$ is an immersion with $A(r)=A(s,t)=J_{2n}$.
As a consequence, $\Sigma^{2n}= \text{Image}(\phi)\subset\mathbb R^{2n+1}$ is an improper affine sphere with Blaschke normal $\xi=(0_{2n},1)$ and Blaschke metric given by
\begin{equation}\label{eq:specialmetric}
h=\left[
\begin{array}{cc}
\frac{\partial^2 Q}{\partial t^2}  & 0 \\
0  &  \frac{\partial^2 Q}{\partial t^2}
\end{array}
\right].
\end{equation}
\end{theorem}
\begin{proof}
The first statement follows from
$$
y_{s_i}=\left( 0, \frac{\partial^2 Q}{\partial s_i\partial s_j}\right)=-\left( 0, \frac{\partial^2 Q}{\partial t_i\partial t_j}\right)=-x_{t_i},
$$
where the center equality comes from Cauchy-Riemann equations. Similarly
$$
y_{t_i}=\left( e_i, \frac{\partial^2 Q}{\partial t_i\partial s_j}   \right)=
\left( e_i, \frac{\partial^2 Q}{\partial s_i\partial t_j}   \right)=x_{s_i}.
$$
Thus by Proposition \ref{prop:Deta},  $\phi$ is an improper affine sphere with Blaschke metric given by equation \eqref{eq:hsym}.
\end{proof}

It is worthwhile to describe this construction in terms of the complex variables $(z,\zb)$, $z=s+it$, $\zb=s-it$. Let
\begin{equation}\label{spcc}
\eta(z,\zb)=x(z,\zb)+iy(z,\zb);\ \ \zeta(z,\zb)=x(z,\zb)-iy(z,\zb).
\end{equation}
Then  $\eta_{\zb}=\zeta_{z}=0$ and so $\eta=\eta(z)$, $\zeta=\zeta(\zb)=\overline{\eta(z)}$.
Moreover, the submanifolds defined by $\eta$ and $\zeta$ are Lagrangian. Finally
\begin{equation}\label{specialcc}
x(z,\zb)=\frac{1}{2}\left( \eta(z)+\zeta(\zb) \right);\ \ y(z,\zb)=\frac{i}{2}\left( \zeta(\zb)-\eta(z) \right).
\end{equation}

We conclude that this immersion has the same algebraic structure as the one in section \ref{sec:Center-chord},  substituting real variables $(s,t)$ by  complex  variables $(z,\zb)$, real immersions $(\beta,\gamma)$ by  holomorphic and anti-holomorphic immersions $(\eta,\zeta)$, with the condition $\zeta=\bar\eta$ (and with the $i$ in (\ref{spcc})-(\ref{specialcc})).
In other words, according to (\ref{complexinvpsi})-(\ref{complexpsi}) and (\ref{spcc})-(\ref{specialcc}), we have:
\begin{proposition}
 Special IAS are naturally related to the rotated center-chord transform of a pair of complex conjugate Lagrangian submanifolds of the complexified (real-$2n$-dimensional) symplectic  space.
\end{proposition}
\begin{remark} In spite of this similarity between these two types of IAS, we will reserve the name {\it center-chord IAS} for the ones related to the center-chord transform of a pair of real Lagrangian submanifolds, keeping the name {\it special IAS} for the complex  case. But we stress the fact that, in both cases, center-chord \emph{and} special, each of these IAS is a \emph{real} hypersurface of $\mathbb R^{2n+1}$. \end{remark}

Thus, when condition (\ref{transvspecial}) is satisfied, the projection $\pi: T\mathbb R^{2n}\to\mathbb R^{2n}$ restricted to $\{(x(s,t),y(s,t)):(s,t)\in U\}=Y_F(V)\subset T\mathbb R^{2n}$ is regular and therefore $ f(s,t)=F(x(s,t))$, where the function $F:V\subset\mathbb R^{2n}\to \mathbb R$ satisfies the classical Monge-Amp\`ere equation, $\det[\partial^2F]=c$.
Singularities of these special IAS occur when condition (\ref{transvspecial})   fails. These singularities will be studied in section \ref{singularsection}.

\subsection{One parameter families}

Given a center-chord IAS $\phi$, there exists a one parameter family of center-chord IAS $\phi_{\lambda}$, $\lambda\in\R,\lambda\neq 0$ with $\phi_{1}=\phi$  such that
$\phi_{\lambda}$ has the same Blaschke metric as $\phi_{1}$, for any $\lambda$, and $\phi_{\lambda_1}$ is not affinely equivalent to $\phi_{\lambda_2}$, if
$\lambda_1\neq \lambda_2$. In fact, take $\beta_{\lambda}(s)=\lambda\beta(s)$ and $\gamma_{\lambda}(s)=|\lambda|^{-1}\gamma(s)$. For
$\lambda=-1$, we get the conjugate IAS. It is not difficult to verify that any center-chord IAS with the same Blaschke metric is affinely equivalent
to some IAS of this family.

Similarly, given a special IAS $\phi$, there exists a one parameter family of center-chord IAS $\phi_{\tau}$, $\tau\in[0,2\pi]$, with $\phi_{0}=\phi$  such that
$\phi_{\tau}$ has the same Blaschke metric as $\phi_{0}$, for any $\tau$, and $\phi_{\tau_1}$ is not affinely equivalent to $\phi_{\tau_2}$, if
$\tau_1\neq \tau_2$. In fact, take $H_{\tau}(z)=e^{2i\tau}H(e^{-i\tau}z)$. For
$\tau=\frac{\pi}{2}$, we get the conjugate IAS. It is not difficult to verify that any special IAS with the same Blaschke metric is affinely equivalent
to some IAS of this family.

In case $n=1$ these results were proved in \cite{Simon93A} for any affine sphere, not necessarily improper.

\subsection{Other examples of even dimensional IAS}\label{eg}

For $n=1$, any IAS is center-chord or special. Next examples show that this is not true for $n>1$.

\begin{example}
Consider $Dx=I_{2n}$ and $a\in sp(2n)$. In this case $f$ is quadratic map. For example, take $n=2$ and
\[
A=\left[
\begin{array}{cccc}
1 & 1 & 0 & 0\\
0 & 1 & 0 & 0\\
0 & 0 & -1 & 0\\
0 & 0 & -1 & -1
\end{array}
\right].
\]
Then $f=x_1x_3+x_2x_4+x_2x_3$ and since $A$ is not similar to $K_{4}$ or $J_{4}$, $(x,f)$ is neither center-chord nor special. Observe that, in this example,
the Blaschke connection $\nabla$ and its dual $\overline{\nabla}$ are flat and thus $A$ is parallel with respect to $\hat{\nabla}$.
\end{example}

\begin{example}
If one considers the product of a center-chord IAS with a special IAS, one obtains a new IAS.
\end{example}

\begin{example}
Consider $n=2$ and $f(x_1,x_2,x_3,x_4)=x_1x_3+x_2x_4+h(x_2,x_3)$. Then
\[
D^2f=\left[
\begin{array}{cccc}
0 & 0 & 1 & 0 \\
0 & h_{x_2x_2} & h_{x_2x_3} & 1 \\
1 & h_{x_2x_3} & h_{x_3x_3} &0 \\
0 & 1 & 0 & 0
\end{array}
\right]
\]
and so $\det(D^2f)=1$. The corresponding matrix $A$ is given by
\[
A=\left[
\begin{array}{cccc}
-1 & -h_{x_2x_3} & -h_{x_3x_3} & 0 \\
0 &  -1 &  0  & 0 \\
0 &  0 &  1 &0 \\
0 & h_{x_2x_2} & h_{x_2x_3} & 1
\end{array}
\right]
\]
and so this IAS is neither center-chord nor special. If $h(x_2,x_3)$ is not quadratic, the dual connection $\overline{\nabla}$ is not flat and it is not difficult to verify
that the matrix $A$ is not parallel with respect to the metric connection $\hat{\nabla}$.
\end{example}

\section{Center-chord and Special IAS as solutions of Exterior Differential Systems }

In this section we shall characterize the center-chord and the special IAS as solutions of certain Exterior Differential Systems (EDS).

\subsection{Center-chord IAS as solutions of an EDS}

Define the involution  ${\mathcal K}_{4n}:\R^{4n}\to\R^{4n}$ by
$$
{\mathcal K}_{4n}(v_1,v_2)=\left( v_2, v_1   \right).
$$
The symplectic form $\Omega_1$ given by \eqref{eq:Omega1} is equivalently defined by
$$
\Omega_1 \left( v, w    \right)=\Omega\left( v,  {\mathcal K}_{4n} w    \right).
$$
Consider the Exterior Differential System ${\mathcal E}_1=\{\Omega,\Omega_1\}$.
The main result of this section is the following:

\begin{theorem}\label{thm:EDS1}
The solutions of the EDS ${\mathcal E}_1$ are exactly the center-chord IAS.
\end{theorem}

We begin with the following lemma:

\begin{lemma}\label{lemma:Center-chordEq}
Consider a $\Omega$-Lagrangian immersion $L$ and denote by $\mathcal{L}$ the image of $U$ by $L$.
The following statements are equivalent:
\begin{enumerate}
\item
$\mathcal{L}$ is $\Omega_1$-Lagrangian, for any $r\in U$.

\item
$\mathcal{L}$ is ${\mathcal K}_{4n}$-invariant, for any $r\in U$.

\item
$A(r)^2=I_{2n}$, for any $r\in U$.

\item
$A(r)$ is equivalent to $K_{2n}$, for any $r\in U$.

\end{enumerate}

\end{lemma}
\begin{proof}
\smallskip\noindent{\bf (1)$\Longleftrightarrow $(2)}.
We start with (1)$\Longrightarrow$(2). Fix $w_0\in T_{L(r)}\mathcal{L}$  and take $w_1,w_2\in T_{L(r)}\mathcal{L}$. Then
$$
\Omega(w_1+\lambda {\mathcal K}_{4n}w_0,w_2+\mu {\mathcal K}_{4n}w_0)=\lambda \Omega_1(w_1,w_0)-\mu \Omega_1(w_2,w_0)=0.
$$
Thus $span\{T_{L(r)}\mathcal{L},{\mathcal K}_{4n}w_0\}$ is $\Omega$-Lagrangian and thus ${\mathcal K}_{4n}w_0\in T_{\phi(r)}\mathcal{L}$, which implies (2). The implication
(2)$\Longrightarrow$(1) is trivial.

\smallskip\noindent{\bf (2)$\Longleftrightarrow $(3)}.
Any vector $(w,w')\in T_{L(r)}\mathcal{L}$ can be written as $w=Dx(r)\cdot u$, $w'=Dx(r)\cdot A(r)\cdot u$.   We have to check when $(w',w)\in T_{L(r)}\mathcal{L}$.

This last condition occurs if and only if $w'=Dx(r)\cdot u_1$, $w'=Dx(r)\cdot A(r)\cdot u_1$, for some $u_1\in T_rU$. But this is equivalent to
$u_1=A(r)\cdot u$, $A(r)u_1=u$. We conclude that $(w',w)$ belongs to $T_{L(r)}\mathcal{L}$ if and only if $A^{-1}(r)=A(r)$, which is equivalent to $A(r)^2=I_{2n}$.

\smallskip\noindent{\bf (3) $\Longleftrightarrow $ (4)}.
It is clear that $A(r)$ similar to $K_{2n}$ implies $A(r)^2=I_{2n}$. Assume now that $A(r)^2=I_{2n}$.
Then the eigenvalues of $A(r)$ are $\pm 1$. Since $(A(r)-I_{2n})\cdot (A(r)+I_{2n})=0$, the minimal polynomial of $A(r)$ contains only linear factors. Hence $A(r)$, and from
equation \eqref{defineB} also $B(r)$, are diagonalizable.

Since $B(r)\in sp(2n)$, $J_{2n}\cdot B(r)=-B^t(r)\cdot J_{2n}$. Take $u$ an eigenvector associated with the eigenvalue $-1$. Then
$B(r)^t\cdot J_{2n}u=J_{2n}u$ and so $J_{2n}u$ is an eigenvector associated with the eigenvalue $+1$. We conclude that the dimensions of the $-1$ and $1$ eigenspaces are equal.
Hence $B(r)$, and thus also $A(r)$, are equivalent to $K_{2n}$.
\end{proof}

The main step in the proof of theorem \ref{thm:EDS1} is the following:

\begin{proposition}\label{prop:Center-chordInv}
Consider an immersion  $\phi:U\subset\R^{2n}\to\R^{2n+1}$ transversal to $(0,1)$ such that the matrix $A(r)$ is equivalent to $K_{2n}$, for any $r\in U$. Then we can realize $\phi$
as a center-chord IAS.
\end{proposition}

\begin{proof}
The matrix $B(r)$ is similar to $K_{2n}$.
Denote by $E_1$ the $-1$-eingenspace and by $E_2$ the $1$-eingenspace.
Let $p_1:x(U)\to\R^{2n}$ be defined as $p_1(x)=x+y(x)$. Then, for any $v_1\in E_1$, $v_2\in E_2$,
$$
Dp_1(x)v_1=v_1+Dy(x)v_1=0;\ \, Dp_1(x)v=v_2+Dy(x)v_2=2v_2.
$$
Thus $p_1$ has rank $n$ at all points. Denoting $\beta=p_1(U)$, observe that the tangent space to $\beta$ at $p_1(x)$ is exactly $E_1$.
For $v_1,w_1\in E_1$,
$$
\omega(v_1,w_1)=-\omega(v_1,K_{2n}w_1)=-\omega(w_1,K_{2n}v_1)=\omega(w_1,v_1).
$$
Thus $\omega(v_1,w_1)=0$ and we conclude that $\beta$ is Lagrangian.
Now consider $p_2:x(U)\to\R^{2n}$ be defined as $p_2(x)=x-y(x)$. As above, $p_2$ has rank $n$ and the tangent space to $\gamma=p_2(U)$
at $p_2(x)$ equals $E_2$. Moreover, $\gamma$ is Lagrangian.
\end{proof}

Now we can prove theorem \ref{thm:EDS1}. If $\phi$ is an immersion such that $\Omega^*L=\Omega_1^*L=0$, then lemma \ref{lemma:Center-chordEq} implies that $A(r)$ is equivalent to
$K_{2n}$, for any $r\in U$. Now proposition \ref{prop:Center-chordInv} implies that $\phi$ can be realized an a center-chord IAS.

\begin{remark} In case $n=1$, any improper affine sphere $\phi:U\subset\R^2\to\R^3$ with indefinite metric necessarily satisfies $A^2(r)=I_{2n}$, for any $r\in U$. Thus, by proposition \ref{prop:Center-chordInv},  $\phi$ can be realized as
a center-chord IAS. In this case, the coordinates $(s,t)$ are called asymptotic (\cite{Craizer11},\cite{Milan13}).
\end{remark}

\subsection{Special IAS as solutions of an EDS }

Consider the complex structure  $\mathcal{J}_{4n}:\R^{4n}\to\R^{4n}$ defined by
$$
\mathcal{J}_{4n}(v_1,v_2)=\left( v_2, -v_1   \right).
$$
The symplectic form $\Omega_2$ defined by \ref{eq:Omega2} is also given by
$$
\Omega_2 \left( v, w    \right)=\Omega\left( v,  \mathcal{J}_{4n} w    \right).
$$
Consider the Exterior Differential System ${\mathcal E}_2=\{\Omega,\Omega_2\}$.
The main result of this section is the following:

\begin{theorem}\label{thm:EDS2}
The solutions of the EDS ${\mathcal E}_2$ are exactly the special IAS.
\end{theorem}

\begin{lemma}\label{lemma:SpecialEq}
Consider a $\Omega$-Lagrangian immersion $L$. The following statements are equivalent:
\begin{enumerate}
\item
$\mathcal{L}$ is $\Omega_2$-Lagrangian, for any $r\in U$.

\item
$\mathcal{L}$ is $\mathcal{J}_{4n}$-invariant, for any $r\in U$.

\item
$A(r)^2=-I_{2n}$, for any $r\in U$.

\item
$A(r)$ is equivalent to $J_{2n}$, for any $r\in U$.

\end{enumerate}
\end{lemma}
\begin{proof}
Similar to lemma \ref{lemma:Center-chordEq}.
\end{proof}

\begin{proposition}\label{prop:SpecialCS}
Consider an immersion  $\phi:U\subset\R^{2n}\to\R^{2n+1}$ transversal to $(0,1)$ such that  $A(r)$ is equivalent to $J_{2n}$, for any $r\in U$. Then $\phi$
can be realized as a special IAS.
\end{proposition}
\begin{proof}
Analogous to proposition \ref{prop:Center-chordInv} using the complex variables $(z,\zb)$.
\end{proof}

Now we can prove theorem \ref{thm:EDS2}. If $\phi$ is an immersion such that $\Omega^*L=\Omega_2^*L=0$, then proposition \ref{lemma:SpecialEq} implies that $A(r)$ is equivalent to $K_{2n}$, for any $r\in U$. Now
proposition \ref{prop:SpecialCS} implies that $\phi$ can be realized an a special IAS.

\begin{remark}
In case $n=1$, any improper affine sphere $\phi:U\subset\R^2\to\R^3$ with definite metric necessarily satisfies $A^2(r)=-I_{2n}$, for any $r\in U$. Thus, by proposition \ref{prop:SpecialCS},  $\phi$ can be realized as
a special IAS. In this case, the coordinates $(s,t)$ are called isothermal
(\cite{Galvez07},\cite{Craizer12},\cite{Martinez05}).
\end{remark}

\section{Lagrangian and Legendrian stable singularities of IAS}\label{singularsection}

Consider a Lagrangian immersion $L:\mathbb R^{2n}\rightarrow (T\mathbb R^{2n},\Omega)$ and a Legendrian immersion $\tilde L:\mathbb R^{2n}\rightarrow (T\mathbb R^{2n}\times \mathbb R,\{\theta=0\})$. Denote by
$\pi:T\mathbb R^{2n}=\R^{2n}\times\R^{2n}\to\R^{2n}$ the projection $\pi(x,y)=x$ and by $\tilde \pi:T\R^{2n}\times\R=\R^{2n}\times\R^{2n}\times\R\to\R^{2n}\times\R$, the projection $\tilde \pi(x,y,z)=(x,z)$.

In this section we shall consider  the singularities of the Lagrangian map $\pi \circ L$ and the Legendrian map $\tilde \pi\circ \tilde L$.

We use the following notation:   $x=(x^{(1)},x^{(2)})=({x^{(1)}}, \hat{x}, \check{x})$, $x^{(1)}=(x_1,...,x_n)$, $\hat{x}=(x_{n+1},...,x_{n+m})$, $\check{x}=(x_{n+m+1},...,x_{2n})$ and
$y=(y^{(1)},y^{(2)})=(\hat{y},\check{y},{y^{(2)}})$, $\hat{y}=(y_{1},...,y_{m})$, $\check{y}=(y_{m+1},...,y_{n})$, $y^{(2)}=(y_{n+1},...,y_{2n})$ for $0\leq m\leq n-1$.

Let us recall that in this notation
\begin{equation}
\Omega=dx^{(1)}\wedge dy^{(2)} +d\hat{y}\wedge d\hat{x}+d\check{y}\wedge d\check{x},
\end{equation}
\begin{equation}
\theta=dz-y^{(2)}dx^{(1)} +\hat{y}d\hat{x}+\check{y}d\check{x}.
\end{equation}

\subsection{ Generating functions and generating families}

The main tool used for classifying these singularities are the generating functions and generating families.

Denote by ${\mathcal L}$ the image of the Lagrangian immersion $L$ and by ${\tilde {\mathcal L}}$ the image of the Legendrian immersion $\tilde L$.

A generating function of the Lagrangian submanifold $\mathcal L$ and the Legendrian submanifold $\tilde {\mathcal L}$ is a function
$$S:\R^{n+m}\times\R^{n-m}\ni ({x^{(1)}},\hat{x},\check{y})\mapsto S({x^{(1)}},\hat{x},\check{y})\in\R,$$ satisfying
\begin{equation}\label{eq:defineGF1}
\mathcal L=\{ (x,y):   \frac{\partial S}{\partial {x^{(1)}}}={y^{(2)}}, \frac{\partial S}{\partial \check{y}}=\check{x}, -\frac{\partial S}{\partial \hat{x}}=\hat{y}  \}.
\end{equation}
and
\begin{equation}\label{eq:defineGF2}
\tilde {\mathcal L}=\{ (x,y,z): \frac{\partial S}{\partial {x^{(1)}}}={y^{(2)}}, \frac{\partial S}{\partial \check{y}}=\check{x}, -\frac{\partial S}{\partial \hat{x}}=\hat{y} , z=S({x^{(1)}},\hat{x},\check{y})-\check{y}\cdot\check{x} \}.
\end{equation}

A generating family of the Lagrangian map $\pi \circ L$ and the Legendrian map $\tilde \pi \circ \tilde L$ is a function
$$G:\R^{2n}\times\R^{n-m}\ni ({x^{(1)}},\hat{x},\check{x},\kappa) \mapsto G({x^{(1)}},\hat{x},\check{x},\kappa)\in \R,$$ satisfying
$$
\mathcal L=\{ (x,y): \exists \kappa : \frac{\partial G}{\partial \kappa}=0,  \frac{\partial G}{\partial {x^{(1)}}}={y^{(2)}}, -\frac{\partial G}{\partial \check{x}}=\check{y} \}.
$$
and
$$
\tilde{\mathcal L}=\{ (x,y,z): \exists \kappa : \frac{\partial G}{\partial \kappa}=0,  \frac{\partial G}{\partial {x^{(1)}}}={y^{(2)}}, -\frac{\partial G}{\partial \check{x}}=\check{y}, z=G({x^{(1)}},\hat{x},\check{x},\kappa) \}.
$$

A generating family can be obtained from a generating function by the formula
\begin{equation}\label{gen-fam-gen-fun}
G({x^{(1)}},\hat{x},\check{x},\kappa)=S({x^{(1)}},\hat{x},\kappa)- \check{x}\cdot\kappa.
\end{equation}

We shall use the following well-known theorem (\cite{Arnold}, Chapter 21):

\begin{theorem}\label{thm-Arnold}
The Lagrangian map-germ $\pi \circ L$ at $0$ generating by $G$ is Lagrangian stable if and only if the classes of function-germs
$$1,\frac{\partial G}{\partial x^{(1)}}(0,0,0,\kappa), \frac{\partial G}{\partial \hat{x}}(0,0,0,\kappa),
\frac{\partial G}{\partial \check{x}}(0,0,0,\kappa)$$
generate the linear space $\mathbb R[[\kappa]]/<\frac{\partial G}{\partial \kappa}(0,0,0,\kappa)>$.

The Legendrian map-germ $\tilde \pi \circ \tilde L$ at $0$ generating by $G$ is Legendrian stable if and only if the classes of function-germs
$$1,\frac{\partial G}{\partial x^{(1)}}(0,0,0,\kappa), \frac{\partial G}{\partial \hat{x}}(0,0,0,\kappa),
\frac{\partial G}{\partial \check{x}}(0,0,0,\kappa)$$
generate the linear space $\mathbb R[[\kappa]] / <\frac{\partial G}{\partial \kappa}(0,0,0,\kappa), G(0,0,0,\kappa)>$.
\end{theorem}

\subsection{Generating functions of center-chord and special IAS}

Consider a center-chord IAS and assume that the Lagrangian submanifolds are given by
$(u, dS^{-}(u))$ and $(v,dS^{+}(v))$.
Then
\begin{eqnarray}
(x^{(1)},x^{(2)})&=&\frac{1}{2}\left( u+v, dS^{+}(v)+dS^{-}(u)\right)\\ (y^{(1)},y^{(2)})&=&\frac{1}{2}\left(v-u,dS^{+}(v)-dS^{-}(u)\right).
\end{eqnarray}
straightforward calculations show that
\begin{equation}\label{ccS_0}
S_0({x^{(1)}},{y^{(1)}})=\frac{1}{2} \left(  S^{+}({x^{(1)}}+{y^{(1)}})-S^{-}({x^{(1)}}-{y^{(1)}})             \right)
\end{equation}
satisfies equations \eqref{eq:defineGF1}-\eqref{eq:defineGF2} and thus is a generating function for $L$ and ${\tilde L}$.
%Thus
%$$
%G({x^{(1)}},{x^{(2)}},{t})=\frac{1}{2} \left(  S^{+}({x^{(1)}}+{t})-S^{-}({x^{(1)}}-{t})   \right) -{t}\cdot {x^{(2)}}.
%$$
%is a generating family for $L$ and ${\tilde L}$.

A special IAS is defined by
\begin{equation}
(x^{(1)},x^{(2)})=(s, \frac{\partial Q}{\partial t})  ;\ \ (y^{(1)},y^{(2)})=(t, \frac{\partial Q}{\partial s}),
\end{equation}
where $ Q$ is the imaginary part of a holomorphic function $H$ (see (\ref{holom-H})-(\ref{hol2})).
Thus
\begin{equation}\label{eq:S_0}
S_0(x^{(1)}, y^{(1)})= Q(x^{(1)}, y^{(1)})
\end{equation}
satisfies equations \eqref{eq:defineGF1} and \eqref{eq:defineGF2} and thus is a generating function.

For any $0\leq m\leq n-1$, it follows from equation \eqref{eq:defineGF1} that
the Lagrangian submanifold of the IAS is defined by the equations
\begin{equation}\label{eq:defineSpecialIAS1}
{y^{(2)}}=\frac{\partial S_0}{\partial {x^{(1)}}}, \ \ \hat{x}=\frac{\partial S_0}{\partial \hat{y}}, \ \  \check{x}=\frac{\partial S_0}{\partial \check{y}}.
\end{equation}
For the Legendrian submanifold, we must consider also
\begin{equation}\label{eq:defineSpecialIAS2}
z=S_0({x^{(1)}}, \hat{y},\check{y})-\hat{y}\cdot \hat{x}- \check{y}\cdot \check{x}.
\end{equation}
So the generating family has the following form
\begin{equation}\label{G_0}
G_0(x^{(1)},\hat{x},\check{x},\kappa)=S_0({x^{(1)}}, \hat{\kappa},\check{\kappa})-\hat{\kappa}\cdot \hat{x}- \check{\kappa}\cdot \check{x},
\end{equation}
where $\kappa=(\hat{\kappa},\check{\kappa})$, $\hat{\kappa}=(\kappa_{1},...,\kappa_{m})$, $\check{\kappa}=(\kappa_{m+1},...,\kappa_{n})$.

We can obtain other generating functions when $S_0$ (given by \eqref{ccS_0} or by \eqref{eq:S_0}) is quadratic in $\hat{y}$. With this assumption we can write
\begin{equation}\label{eq:defineVquadratic}
S_0({x^{(1)}},\hat{y},\check{y})=\sum_{k=1}^m\left(\frac{1}{2}y_k^2+y_kg_k({x^{(1)}},\check{y})\right) +h({x^{(1)}},\check{y}).
\end{equation}
Then from $\hat{x}=\frac{\partial S_0}{\partial \hat{y}}$ we obtain
\begin{equation}\label{y_k}
 y_k=x_{k+n}-g_k({x^{(1)}},\check{y})
\end{equation}
 for $k=1,\cdots,m$. Let $\hat{g}({x^{(1)}},\check{y})=(g_1({x^{(1)}},\check{y}),\cdots, g_m({x^{(1)}},\check{y}))$. We define a new generating function
\begin{equation}\label{S_m}
S_m({x^{(1)}},\hat{x},\check{y})=S_0({x^{(1)}},\hat{x}-\hat{g}({x^{(1)}},\check{y}),\check{y}) - \sum_{k=1}^m (x_{k+n}-g_k({x^{(1)}},\check{y}))x_{k+n},
\end{equation}
Thus
\begin{equation}\label{eq:Gfm}
S_m({x^{(1)}},\hat{x},\check{y})=-\frac{1}{2}\sum_{k=1}^m\left(x_{k+n}- g_k({x^{(1)}},\check{y}) \right)^2+h({x^{(1)}},\check{y}).
\end{equation}
Using equations \eqref{eq:defineSpecialIAS1} and \eqref{eq:defineSpecialIAS2}, it is straightforward to verify that $S_m$
satisfies \eqref{eq:defineGF1} and \eqref{eq:defineGF2} and thus is a generating function for the Lagrangian and Legendrian submanifolds. From \eqref{gen-fam-gen-fun} we obtain a new generating family
\begin{equation}\label{G_m}
G({x^{(1)}},\hat{x},\check{x},\check{\kappa})=S({x^{(1)}},\hat{x},\kappa)- \check{x}\cdot\check{\kappa}.
\end{equation}

\subsection{Realization of simple stable Legendrian singularities of center-chord IAS}

The singular set of a center-chord IAS has a simple geometrical meaning: The tangent planes of the Lagrangian submanifolds
 $(s, dS^{-}(s))$ and $(t, dS^{+}(t))$ at a singular pair $(s,t)$ must intersect in a $(n-m)$-dimensional vector space, with $n>m$.
 Thus the image of the singular set by the map $x$ is exactly their Wigner caustic.
Moreover, the image of the singular set by the one parameter family $x_{\lambda}$ are the equidistants of the pair of Lagrangian submanifolds.

The singularities of the equidistants and the Wigner caustic of a pair of Lagrangian submanifolds were studied in \cite{Domitrz13}, so
the problem of realization of simple singularities for center-chord IAS was solved there, where it is proved that
any simple stable Lagrangian singularity can be realized by a center-chord transform and, as
it is straightforward to adapt section $4$ of this paper to the Legendrian setting,
Theorem $4.1$ in \cite{Domitrz13} can be restated as follows:

\begin{theorem}
Any germ of a simple stable Legendrian singularity is realizable as a center-chord IAS.
\end{theorem}

We give below another proof of this theorem, by presenting new generating families that closely resemble the generating families that appear in the proof of Theorem \ref{specialsing}  for the special IAS, in the next section, thus re-emphasizing the similarities between these two kinds of IAS.

\begin{proof}
We explicitly describe pairs of functions $(S^{+},S^{-})$ such that the corresponding generating function $S_0$ (see \eqref{ccS_0}) or $S_m$ (see \eqref{S_m}) generates the simple stable Legendrian singularity of the type A-D-E.
Using formulas \eqref{G_0} and \eqref{G_m} and Theorem \ref{thm-Arnold} one can easily check that for the following pair of functions $(S^{+},S^{-})$ we obtain the following Legendrian singularities:

Denote by $\lfloor a \rfloor$ the greatest integer smaller than or equal to $a$.

For $A_k$ singularity with $k\le n+2$ we take
$$
S^{+}(v)=\pm (-1)^{  \lfloor \frac{k+1}{2} \rfloor } v_1^{k+1} + \sum_{j=2}^n v_j^2+\sum_{j=1}^{k-2} (-1)^{ \lfloor \frac{k-j+1}{2} \rfloor} v_jv_1^{k-j}
$$
$$
S^{-}(u)=\pm (-1)^{  \lfloor \frac{k}{2} \rfloor } u_1^{k+1} - \sum_{j=2}^n u_j^2+\sum_{j=1}^{k-2} (-1)^{ \lfloor \frac{k-j+2}{2} \rfloor} u_ju_1^{k-j}
$$
Then a generating function is $S_0$ (see \eqref{ccS_0}).

For $A_k$ singularity with $n+2<k< 2n+2$ we take
$$
S^{+}=(-1)^{\lfloor \frac{k+1}{2} \rfloor} v_1^{k+1}+\sum_{j=2}^{k-n-1}  (-1)^{\lfloor \frac{k-j}{2} \rfloor} v_jv_1^{k-j+1}+\sum_{j=1}^{n}  (-1)^{\lfloor \frac{n-j+2}{2} \rfloor} v_{j}v_1^{n-j+2}
$$
$$
+\frac{1}{2}\sum_{j=2}^{k-n-1}v_1^{2k-2j+2} +\frac{1}{2}\sum_{j=2}^{k-n-1}v_j^2 +\sum_{j=k-n}^{n}v_j^2.
$$
$$
S^{-}=(-1)^{\lfloor \frac{k}{2}\rfloor} u_1^{k+1}+\sum_{j=2}^{k-n-1} (-1)^{\lfloor \frac{k-j-1}{2} \rfloor} u_ju_1^{k-j+1}+\sum_{j=1}^{n}  (-1)^{\lfloor \frac{n-j+1}{2} \rfloor} u_{j}u_1^{n-j+2}
$$
$$
-\frac{1}{2}\sum_{j=2}^{k-n-1}u_1^{2k-2j+2} -\frac{1}{2}\sum_{j=2}^{k-n-1}u_j^2 -\sum_{j=k-n}^{n}u_j^2.
$$
Then $S_0$ (see \eqref{ccS_0}) is quadratic in $y_j=v_j-u_j$ for $j=2,\cdots,k-n-1$. So we can  use  \eqref{eq:defineVquadratic}-\eqref{eq:Gfm} to obtain a generating function $S_m$.

For $D_k^{\pm}$ with $4\le k\le n+3$ take
$$
S^{+}(v)=-v_1v_2^2-v_1^3-v_2v_1^3\pm (-1)^{ \lfloor \frac{k+3}{2} \rfloor}v_1^{k-1}+\sum_{j=3}^{k-3}(-1)^{ \lfloor \frac{j+2}{2} \rfloor}v_jv_1^{j+1}+\sum_{j=3}^{n} v_j^2.
$$
$$
S^{-}(u)=-u_1u_2^2+u_1^3-u_2u_1^3\pm (-1)^{ \lfloor \frac{k+2}{2} \rfloor}u_1^{k-1}+\sum_{j=3}^{k-3}(-1)^{ \lfloor \frac{j+3}{2} \rfloor}u_ju_1^{j+1}-\sum_{j=3}^{n} u_j^2.
$$
Then a generating function is $S_0$ (see \eqref{ccS_0}).

For $D_k$ with $n+3<k<2n+2$ take
$$
S^{+}=(-1)^{\lfloor \frac{k-1}{2} \rfloor}v_1^{k-1}-v_1v_2^2-v_1^3-v_2v_1^3+\sum_{j=3}^{n}(-1)^{\lfloor \frac{j+1}{2} \rfloor}v_jv_1^{j+1}
$$
$$
+\sum_{j=3}^{k-n-1}(-1)^{\lfloor \frac{j+n}{2} \rfloor}v_{j}v_1^{j+n-1}+\frac{1}{2}\sum_{j=3}^{k-n-1}(-1)^{j+n+1}v_1^{2j+2n-2} -\frac{1}{2}\sum_{j=3}^{k-n-1}v_j^2 -\sum_{j=k-n}^{n}v_j^2.
$$
$$
S^{-}=(-1)^{\lfloor \frac{k-2}{2}\rfloor}u_1^{k-1}-u_1u_2^2+u_1^3-u_2u_1^3+\sum_{j=3}^{n}(-1)^{\lfloor \frac{j}{2} \rfloor}u_ju_1^{j+1}
$$
$$
+\sum_{j=3}^{k-n-1}(-1)^{\lfloor \frac{j+n-1}{2}\rfloor}u_{j}u_1^{j+n-1}+\frac{1}{2}\sum_{j=3}^{k-n-1}(-1)^{j+n}u_1^{2j+2n-2} +\frac{1}{2}\sum_{j=3}^{k-n-1}u_j^2 +\sum_{j=k-n}^{n}u_j^2.
$$
Then $S_0$ (see \eqref{ccS_0}) is quadratic in $y_j=v_j-u_j$ for $j=3,\cdots,k-n-1$. So we can  use  \eqref{eq:defineVquadratic}-\eqref{eq:Gfm} to obtain a generating function $S_m$.

For $E_6^{\pm}$ and $n\ge 3$,  take
$$
S^{+}=-v_1^3 \pm v_2^4+v_3^2+v_1^2v_2^2+v_2^2v_1+v_3v_2^2+\sum_{j=4}^n v_j^2
$$
$$
S^{-}=-u_1^3 \mp u_2^4-u_3^2+u_1^2u_2^2-u_2^2u_1-u_3u_2^2-\sum_{j=4}^n u_j^2
$$
and take the generating function $S_0$ (see \eqref{ccS_0}).

For $E_7$, $n=3$, take
$$
S^{+}=-v_1^3+v_1v_2^3-\frac{1}{2}v_3^2+v_3v_2^4+v_2^4+v_1v_2^2+v_3v_1v_2+\frac{1}{2}v_2^8.
$$
$$
S^{-}=-u_1^3-u_1u_2^3+\frac{1}{2}u_3^2+u_3u_2^4+u_2^4-u_1u_2^2-u_3u_1u_2-\frac{1}{2}u_2^8.
$$
Then a generating function $S_0$(see \eqref{ccS_0}) is quadratic in $y_3=v_3-u_3$ and we can use \eqref{eq:defineVquadratic}-\eqref{eq:Gfm} to obtain a generating function $S_m$.

For $E_7$, $n\geq 4$, we take
$$
S^{+}=-v_1^3+v_1v_2^3-v_3^2-v_4^2+v_2^4v_1+v_2^4+v_3v_2^2+v_4v_1v_2+\sum_{j=5}^n  v_j^2
$$
$$
S^{-}=-u_1^3-u_1u_2^3+u_3^2+u_4^2-u_2^4u_1+u_2^4-u_3u_2^2-u_4u_1u_2-\sum_{j=5}^n  u_j^2
$$
to obtain a generating function $S_0$ (see \eqref{ccS_0}).

For $E_8$, $n=4$, we take
$$
S^{+}=-v_1^3+v_2^5-v_3^2-\frac{1}{2}v_4^2+v_1v_2^3v_4+v_2^4+v_1^2v_2^2+v_1v_2v_3+v_2^2v_4+\frac{1}{2}v_1^2v_2^6.
$$
$$
S^{-}=-u_1^3+u_2^5+u_3^2+\frac{1}{2}u_4^2+u_1u_2^3u_4+u_2^4+u_1^2u_2^2-u_1u_2u_3-u_2^2u_4-\frac{1}{2}u_1^2u_2^6.
$$
Then a generating function $S_0$ (see \eqref{ccS_0}) is quadratic in $y_4=v_4-u_4$ and we can  use  \eqref{eq:defineVquadratic}-\eqref{eq:Gfm} to obtain a generating function $S_m$.

For $E_8$, $n\geq 5$, we take
$$
S^{+}=-v_1^3+v_2^5-v_3^2-v_4^2-v_5^2+v_1^2v_2^3+v_2^4+v_3v_1v_2^2+v_4v_1v_2+v_5v_2^2+\sum_{j=6}^n v_j^2
$$
$$
S^{-}=-u_1^3+u_2^5+u_3^2+u_4^2+u_5^2-u_1^2u_2^3+u_2^4+u_3u_1u_2^2-u_4u_1u_2-u_5u_2^2-\sum_{j=6}^n  u_j^2
$$
to obtain a generating function $S_0$ (see \eqref{ccS_0}).
\end{proof}

We remark that if an IAS is modeled by the center-chord transform of the same Lagrangian submanifold $L\times L\subset V\times V$, then another kind of singularity of the Wigner caustic appears in the limit of vanishing chords, the so-called Wigner caustic on shell. These singularities of the Wigner caustic on shell, which are close to and include the shell $L$, differ from the singularities of the Wigner caustic off shell because, in the former case, their generating families are necessarily {\it odd} functions of their variables, so they are symmetric singularities, in this sense. These symmetric singularities that are realized for IAS of this kind (center-chord transform of $L\times L$) have been studied in \cite{DMR}.

\subsection{Realization of simple stable Legendrian singularities as special IAS}

In this section we show that any simple singularity can be realized as a special IAS, which is a main result of the paper.

\begin{theorem}\label{specialsing}
Any germ of a simple stable Legendrian singularity is realizable as a special IAS.
\end{theorem}

\begin{proof}

We explicitly describe holomorphic functions $H$ (see (\ref{holom-H})) such that the corresponding generating function $S_0=\text{Im} H$ or $S_m$ (see \eqref{S_m}) generates the simple stable Legendrian singularity of the type A-D-E.
Using formulas \eqref{G_0} and \eqref{G_m} and Theorem \ref{thm-Arnold} one can easily check that for the following holomorphic functions $H$ we obtain the following Legendrian singularities:

For $A_k$ singularity with $k\le n+2$ we take
$$
H(z)=\pm i^{3k}z_1^{k+1} - \sum_{j=2}^n  iz_j^2+\sum_{j=1}^{k-2} i^{3k+j+1}z_jz_1^{k-j}.
$$
Then a generating function is $S_0=\text{Im}H$.

For $A_k$ singularity with $n+2<k<2n+2$ we take
$$
H(z)=\pm i^{3k}z_1^{k+1} + \sum_{j=2}^{k-n-1}i^{3k+j+1}z_jz_1^{k-j+1}+ \sum_{j=1}^{n}i^{j-n-1}z_{j}z_1^{n-j+2}
$$
$$
+\frac{1}{2}\sum_{j=2}^{k-n-1}i^{2k+2j+3}z_1^{2k-2j+2}-\frac{1}{2}\sum_{j=2}^{k-n-1}iz_j^2-\sum_{j=k-n}^{n} iz_j^2
$$
Then a generating function $S_0=\text{Im}H$ is quadratic in $y_j=\text{Im}z_j$ for $j=2,\cdots,k-n-1$. So we  use  \eqref{eq:defineVquadratic}-\eqref{eq:Gfm} to obtain a generating function $S_m$.

For $D_k^{\pm}$ with $4\le k\le n+3$ take
$$
H(z)=-z_1z_2^2-iz_1^3-z_2z_1^3\pm i^{3k+2}z_1^{k-1}+\sum_{j=3}^{k-3}i^{3j}z_jz_1^{j+1}-\sum_{j=3}^{n} iz_j^2.
$$
Then a generating function is $S_0=\text{Im}H$.

For $D_k^{\pm}$ with $n+3<k<2n+2$ we take
$$
H(z)=-z_1z_2^2-iz_1^3-z_2z_1^3\pm i^{3k+2}z_1^{k-1}+\sum_{j=3}^{n}i^{3j}z_jz_1^{j+1}+\sum_{j=3}^{k-n-1}i^{3j+3n+1}z_jz_1^{j+n-1}
$$
$$
+\frac{1}{2}\sum_{j=3}^{k-n-1}i^{2j+2n+3}z_1^{2j+2n-2}-\frac{1}{2}\sum_{j=3}^{k-n-1}iz_j^2-\sum_{j=k-n}^{n} iz_j^2.
$$
Then a generating function $S_0=\text{Im}H$ is quadratic in $y_j=\text{Im}z_j$ for $j=3,\cdots,k-n-1$. So we use \eqref{eq:defineVquadratic}-\eqref{eq:Gfm} to obtain a generating function $S_m$.

For $E_6^{\pm}$ and $n\ge 3$,  we take
$$
H(z)=-z_1^3\pm iz_2^4+iz_3^2+z_1^2z_2^2+iz_2^2z_1+iz_3z_2^2-\sum_{j=4}^n  i z_j^2
$$
to obtain a generating function $S_0=\text{Im}H$.

For $E_7$, $n=3$, we take
$$
H(z)=-z_1^3+iz_1z_2^3-\frac{1}{2}iz_3^2+z_3z_2^4+z_2^4+iz_1z_2^2+iz_1z_2z_3+\frac{1}{2}iz_2^8
$$
Then a generating function $S_0=\text{Im}H$ is quadratic in $y_3=\text{Im}z_3$ and we can use \eqref{eq:defineVquadratic}-\eqref{eq:Gfm} to obtain a generating function $S_m$.

For $E_7$, $n\geq 4$, we take
$$
H(z)=-z_1^3+iz_1z_2^3-iz_3^2-iz_4^2+iz_1z_2^4+z_2^4+iz_2^2z_3+iz_1z_2z_4-\sum_{j=5}^n  i z_j^2
$$
to obtain a generating function $S_0=\text{Im}H$.

For $E_8$, $n=4$, we take
$$
H(z)=-z_1^3+z_2^5-iz_3^2-\frac{1}{2}iz_4^2+z_1z_2^3z_4+z_2^4+z_1^2z_2^2+iz_1z_2z_3+iz_2^2z_4+\frac{1}{2}iz_1^2z_2^6
$$
Then a generating function $S_0=\text{Im}H$ is quadratic in $y_4=\text{Im}z_4$. So we can  use  \eqref{eq:defineVquadratic}-\eqref{eq:Gfm} to obtain a generating function $S_m$.

For $E_8$, $n\geq 5$, we take
$$
H(z)=-z_1^3+z_2^5-iz_3^2-iz_4^2-iz_5^2+iz_1^2z_2^3+z_2^4+z_1z_2^2z_3+iz_1z_2z_4+iz_2^2z_5-\sum_{j=6}^n  i z_j^2
$$
to obtain a generating function $S_0=\text{Im}H$.
\end{proof}

% ------------------------------------------------------------------------
\end{document}